\documentclass[10pt,reqno]{amsart}
\usepackage[cp1251]{inputenc}
\usepackage[english]{babel}
\usepackage{amsmath}
\usepackage{amssymb}
\usepackage{amsfonts}
\usepackage{graphicx}
\usepackage{eucal}
\usepackage{hyperref,amsthm}
\RequirePackage{bbold}
\usepackage{mathptmx,bm}
\newtheorem{theorem}{Theorem}
\newtheorem{lemma}{Lemma}
\newtheorem*{proposition}{Proposition}

\theoremstyle{definition}

\newtheorem{remark}{Remark}
\numberwithin{equation}{section}
\begin{document}

\title[On invariant measure of Galton-Watson Branching Systems]
    {Further remarks on the explicit generating function expression of the invariant measure of critical Galton-Watson Branching Systems}

\author{{Azam~A.~Imomov}}
\address {Azam Abdurakhimovich Imomov
\newline\hphantom{iii} Karshi State University,
\newline\hphantom{iii} 17, Kuchabag st., 180100 Karshi city, Uzbekistan.}
\email{{imomov{\_}\,azam@mail.ru} \qquad ORCID: {https://orcid.org/0000-0003-1082-0144}}

\thanks{\copyright \ 2023 Imomov~A.A}




\begin{abstract}
    {Consider the critical Galton-Watson branching system with infinite variance of the offspring law. We provide an alternative
    arguments against what Slack~{\cite{Slack68}} did when it seeked for a local expression in the neighborhood of point $1$ of the
    generating function for invariant measures of the branching system. So, we obtain the global expression for all $s\in[0,1)$ of
    this generating function. A fundamentally improved version of the differential analogue of the Basic Lemma of the theory of
    critical branching systems is established. This assertion plays a key role in the formulation of the local limit theorem
    with explicit terms in the asymptotic expansion of local probabilities. We also determine the decay rate of the remainder
    term in this expansion.}
\end{abstract}

\maketitle

\section{Background, assumptions and purpose}

    Let ${Z_n}$ be the population size in the Galton-Watson Branching (GWB) System at time $n\in \mathbb{N}_0$, where
    $\mathbb{N}_0 = \left\{0 \right\} \cup \mathbb{N}$ and $\mathbb{N}= \left\{{1,2, \ldots} \right\}$. An evolution of the system
    will occur by the following scheme. Each individual lives a unit lifespan and at the end of his life produces $j$ progeny with
    probability $p_j$, $j\in \mathbb{N}_0$, independently of each other at that $p_0 >0$. Newborn individuals subsequently undergo
    reproduction obeying the offspring law $\left\{{p_j}\right\}$. The population sizes sequence can be represented by the
    following recurrent random sum of random variables:
\begin{equation}                    \label{1.1}
        Z_{n+1} = \xi_{n1}+\xi_{n2}+ \cdots + \xi_{nZ_{n}}
\end{equation}
    for any $n\in \mathbb{N}$, where $\xi_{nk}$ are independent random variables with the common distribution
\[
    \mathbb{P}\left\{{\xi_{nk}=j}\right\}=p_{j}
\]
    for all $k\in\mathbb{N}$. These variables are interpreted
    as the number of descendants of the $k$th individual in the $n$th generation.
    The GWB system defined above forms a reducible, homogeneous-discrete-time Markov chain with a state space
    consisting of two classes: $\mathcal{S}_{0}=\left\{0\right\}\cup{\mathcal{S}}$, where $\mathcal{S}\subset{\mathbb{N}}$,
    therein the state $\left\{0\right\}$ is absorbing, and $\mathcal{S}$ is the class of possible essential communicating
    states. Its $n$-step transition probabilities
\[
    P_{ij}{(n)} := \mathbb{P}\left\{{{Z_{n+k}=j} \bigm\vert Z_k =i}\right\}
    \qquad \text{for any}\quad k \in \mathbb{N}
\]
    are completely given by the  offspring law $\left\{{p_j}\right\}$. In fact, denoting
    ${\textsf{p}_{j}{(n)}}:={P_{1j}{(n)}}$, we observe that a probability Generating Function (GF)
\[
    \sum_{j \in {\mathcal{S}_{0}}}{P_{ij}{(n)}s^j} = \bigl[{f_n(s)}\bigr]^{i}
\]
    for any $i\in\mathcal{S}$ and $s \in [0, 1)$, where
\[
    f_n(s)=\sum_{j \in {\mathcal{S}_{0}}}{\textsf{p}_{j}{(n)}s^j}.
\]
    At that  the GF $f_n(s)$ is the $n$-fold iteration of the GF ${f}(s):= \sum_{j \in {\mathcal{S}_{0}}}{p_{j}s^j}$.

    The classification of $\mathcal{S}$ depends on the value of the parameter
\[
    m :=\sum_{j \in{\mathcal{S}}}{jp_j}= f'(1-),
\]
    the mean per-capita offspring number. The chain $\left\{{Z_n} \right\}$ is classified as sub-critical, critical and
    supercritical if $m < 1$, $m = 1$ and $m > 1$ respectively. Needless to say that $f_n(0)=\textsf{p}_0(n)$ is a vanishing
    probability of the system initiated by one individual, which is monotone and $\lim_{n\to\infty}\textsf{p}_0(n)=q$, where
    $q$ is called an extinction probability of the system and it is smallest nonnegative root of the fixed-point equation
    $f(q)=q$ on the domain of $\{s: s\in[0,1]\}$. Furthermore $f_n(s)\to{q}$ as $n\to\infty$ uniformly in $s\in[0,1)$;
    see \cite[Ch.I,~\S\S1--5]{ANey}.

    In the paper we focus on the critical case in which $q=1$.
    We assume that the offspring GF $f(s)$ for $s \in [0, 1)$ has the following form:
\[            \label{RAf}
    f(s)=s+(1-s)^{1+\nu} \mathcal{L}\left( {\frac{1}{1-s}}\right),   \eqno[\textsf {$f_\nu$}]
\]
    where $0 < \nu < 1$ and $\mathcal{L}(\ast)$ is slowly varying (SV) function at infinity. By the criticality
    of the system, the assumption {\hyperref[RAf]{$\left[{f_\nu}\right]$}} implies that $2b:=f''(1-)=\infty$.
    If $0<b <\infty$ then $\nu =1$ and $\mathcal{L}(t) \to {b}$ as $t \to \infty$.

    Further, putting into practice the function
\[
    \Lambda (y):={\frac{f(1-y)-(1-y)}{y}} = y^\nu \mathcal{L}\left({\frac{\,1\,}{y}}\right)
\]
    for $y\in(0,1]$, we rewrite and will use the condition {\hyperref[RAf]{$\left[{f_\nu}\right]$}} in the following form:
\[            \label{RAfL}
    f(s)-s=(1-s)\Lambda(1-s).             \eqno[\textsf {$f_\Lambda$}]
\]

    Slack~{\cite{Slack68}} has shown (see, also {\cite{Seneta67}}) that
\begin{equation}              \label{1.2}
    \bar{U}_n(s):= {\frac{f_n(s)-f_n(0)}{f_n(0)-f_{n-1}(0)}} \longrightarrow U(s)
    \qquad \text{as}\quad n \to \infty
\end{equation}
    for $s \in [0, 1]$, where the limit function $U(s)$ is the GF of the invariant measure
    for the system $\left\{{Z_n}\right\}$, and it satisfies the Abel equation
\begin{equation}              \label{1.3}
    U\bigl( f(s) \bigr) = U(s) + 1.
\end{equation}
    Moreover, in the case when {\hyperref[RAf]{$\left[{f_\nu}\right]$}} attends SV-function ${L}(\ast)$ at zero
    instead of $\mathcal{L}(\ast)$, Slack~{\cite{Slack68}} found that $U(s)$ admits a local expression
\[
    U(s)\sim\frac{1}{\nu(1-s)^{\nu}{L}(1-s)} \qquad \text{as}\quad s\uparrow{1}.
\]

    The mean value theorem implies that $f_{n+1}(0)-f_{n}(0) \sim f_{n}(0)-f_{n-1}(0)$ as $n\to\infty$
    and hence altering Slack's definition of $\bar{U}_n(s)$ to
\[              \label{SAU}
    U_n (s):= {\frac{f_n (s)-f_n (0)}{f_{n+1}(0)-f_{n}(0)}} \raise1.2pt\hbox{,}  \eqno[\textsf {$\mathcal{S}_{U}$}]
\]
    we se  that $\lim_{n\to\infty}{U_n (s)}=U(s)$. Then Slack's~{\cite{Slack68}} arguments, in contrast to the method
    in {\cite{Imomov19}}, made it easy to prove the following statement, called the Basic Lemma of the theory of
    critical GWB systems, which clearly shows an explicit asymptotic expression for the function
\[
    {R_{n}(s)}:= 1-f_n(s).
\]

\begin{lemma}[\textbf{Basic Lemma}~{\cite{ImomovSM23}}]                 \label{AALem1}
    If the condition {\hyperref[RAfL]{$\left[{f_\Lambda}\right]$}} holds then
\begin{equation}              \label{1.4}
    R_n(s)={\frac{\mathcal{N}\left(n \right)}{\left(\nu{n}\right)^{{1/\nu}}}} \cdot
    \left[{1-{\frac{\mathcal{U}_n(s)}{\nu{n}}}} \right],
\end{equation}
    where the function $\mathcal{N}(x)$ is SV at infinity and
\begin{equation}              \label{1.5}
    \mathcal{N}(n) \cdot \mathcal{L}^{{1/\nu}} \left({\frac{(\nu{n})^{{1/\nu}}}{\mathcal{N}(n)}} \right) \longrightarrow 1
    \qquad \text{as}\quad n \to \infty,
\end{equation}
    and the function $\mathcal{U}_n (s)$ has the following properties:
\begin{itemize}
\item  $\mathcal{U}_n (s) \to {U}(s)$ as $n \to \infty$ so that the equation \eqref{1.3} holds;
\item  $ \lim _{s \uparrow 1} \mathcal{U}_n (s) = \nu n$ for each  fixed $n \in \mathbb{N}$;
\item  $ \mathcal{U}_n (0) = 0$ for each  fixed $n \in \mathbb{N}$.
\end{itemize}
\end{lemma}

    First direct result of the statement of the Basic Lemma~\ref{AALem1} is certainly
    an expression of survival probability of the family of one individual in a form of
\[
    {Q_n}:=\mathbb{P}\left\{{Z_{n}>0} \right\}={R_n(0)}=
    {\frac{\mathcal{N}(n)}{\left(\nu{n}\right)^{1/\nu}}}\raise1.2pt\hbox{.}
\]

    In our discussions an important role plays the following assertion,
    which we call the differential analogue of the Basic Lemma~\ref{AALem1}.
\begin{lemma}[{\cite{ImomovSM23}}]                 \label{AALem2}
    Let the condition {\hyperref[RAfL]{$\left[{f_\Lambda}\right]$}} holds. Then the following relation is true:
\begin{equation}              \label{1.6}
    {R'_n (s)} = -\psi_{n}(s)
    {\frac{R_n(s)\Lambda \bigl({R_n(s)}\bigr)}{(1-s) \Lambda (1-s)}} \raise1.2pt\hbox{,}
\end{equation}
    where the function $\psi_n(s)$ has following properties:
\begin{itemize}
\item  $\psi_n(s)$ is continuous in $s \in [0, 1)$, for all $n \in\mathbb{N}$ and
\[
    {\frac{f'(s)}{f'\bigl(f_{n}(s)\bigr)}} \leq \psi_{n}(s) \leq 1;
\]
\item  $\psi(s):=\lim _{n \to \infty} \psi_n(s)$ exists for $s \in [0, 1)$ and
\begin{equation*}
    f'(s) \leq \psi(s) \leq 1  \qquad  \mbox{and} \qquad  \psi(1-) = 1.
\end{equation*}
\end{itemize}
\end{lemma}

    By definition  $\mathcal{L}\left({\lambda{x}}\right)\big/{\mathcal{L}(x)}\to{1}$ as $x\to\infty$ for each $\lambda >0$.
    Then it is natural that
\[
    \alpha_{\lambda}(x):={\frac{\mathcal{L}\left({\lambda{x}}\right)}{\mathcal{L}(x)}} - 1
\]
    decreases to zero with a certain speed rate at infinity. With a known rate of decrease of $\alpha_{\lambda}(x)$,
    the function $\mathcal{L}(\ast)$ is called SV at infinity with remainder; see {\cite[p.~185]{Bingham}}.

    The following statement is also known, which is an improved analogue of the Basic Lemma~\ref{AALem1}.
\begin{lemma}[{\cite{ImomovTukhtaev19}}]            \label{AALem3}
    Let the condition {\hyperref[RAfL]{$\left[{f_\Lambda}\right]$}} holds and
    $\alpha_{\lambda}(x)={o}\left({{{\mathcal{L}\left(x\right)}\big/{x^\nu}}}\right)$. Then
\begin{equation}              \label{1.7}
    {\frac{1}{{\Lambda}\bigl({R_{n}(s)}\bigr)}}- {\frac{1}{\Lambda(1-s)}} =
    \nu n +{\frac{1+\nu}{2}}\cdot \ln \Bigl[{\Lambda(1-s)}\nu{n}+1 \Bigr] + \rho_n(s),
\end{equation}
    {\it where $\rho_n(s)={o}\bigl(\ln n\bigr) + \sigma_n(s)$ and, $\sigma_n(s)$
    is bounded uniformly for $s \in [0, 1)$ and converges to the limit $\sigma(s)$
    as $n \to \infty $ which is a bounded function for all $s \in [0, 1)$.}
\end{lemma}

    The peculiarity of the Lemma~{\ref{AALem2}} is that it perfectly generalizes an analogous statement established
    in {\cite[Theorem~1]{KNS}}, in which the offspring law variance was assumed to be finite and later refined under a third
    finite moment assumption in {\cite[p.~20]{Harris63}}. In both papers just mentioned, ${\nu}=1$ and ${\Lambda}(y)\equiv{y}$,
    and thereat ${f''(1-)n}\big/2$ appeared instead of the first term $\nu{n}$ and moreover, the subsequent tail terms
    are found on the right-hand side of \eqref{1.6}.

    In accordance with our purpose, we now recall the following theorem,
    which shows the explicit-integral form of the invariant measure GF $U(s)$.
\begin{theorem}[{\cite{ImomovSM23}}]               \label{AATh1}
    If condition {\hyperref[RAfL]{$\left[{f_\Lambda}\right]$}} holds and
    $\alpha_{\lambda}(x)={o}\left({{{\mathcal{L}\left(x\right)}\big/{x^\nu}}}\right)$, then
\begin{itemize}
\item [\textbf{(1)}] the GF $U(s)$ is
\begin{equation}                 \label{1.8}
    U(s) = \int_0^s {{\frac{\psi(y)} {(1-y) \Lambda (1-y)}}\,dy},
\end{equation}
    where $\psi(s)$ is continuous in $s \in [0, 1]$, and
\begin{equation*}
    f'(s) \leq \psi(s) \leq 1;
\end{equation*}
\item [\textbf{(2)}] the derivative $U'(s)$ has the following representation:
\begin{equation}                 \label{1.9}
    U'(s) = {{\psi{(s)}}\over {(1-s) \Lambda (1-s)}} \raise1.2pt\hbox{,}
\end{equation}
    where $\psi{(s)}=1+\mathcal{O}\bigl(\Lambda{(1-s)}\bigr)$ as $s \uparrow{1}$.
\end{itemize}
\end{theorem}

    In the last statements inequality estimations for the functions $\psi_n(s)$
    and $\psi(s)$ were announced, but explicit expressions were not obtained for them.

    In this paper, in addition to the assumption {\hyperref[RAfL]{$\left[{f_\Lambda}\right]$}},
    we adopt the remainder term rate of the SV-function $\mathcal{L}(\ast)$ to be
\[            \label{RAO(L/x)}
    \alpha_{\lambda}(x) = \mathcal{O}\left( {{{\mathcal{L}\left(x\right)} \over {x^\nu}}} \right)
    \qquad \text{as}\quad x \to \infty,     \eqno[\textsf {$\mathcal{R}_{\mathcal{L}}$}]
\]
    that is more exact decreasing speed rate condition, than it was
    assumed in contents of the Lemma~\ref{AALem3} and in the Theorem~\ref{AATh1}.

    Our purpose is as follows. First, we improve the result of Theorem~\ref{AATh1} by finding an explicit
    expression for the function $U(s)$ that is more exactly than in \eqref{1.8} and an explicit expression
    for the ``undesirable'' function $\psi(s)$ in the equality \eqref{1.9} depending on GF $f(s)$ and $f'(s)$.
    This contributes to the refinement of the formula \eqref{1.6}, pointing to the explicit form of the function
    $\psi_n(s)$. In this issue we propose another proof of the Lemma~\ref{AALem2} that improves its content.
    Next, using condition {\hyperref[RAO(L/x)]{$\left[\mathcal{R}_{\mathcal{L}}\right]$}}, we find the main
    part term in the asymptotic expansion of the right-hand side of \eqref{1.6} with an estimate for the
    remainder term. All these results facilitate to refine some limit theorems.

    The rest of this paper is organized as follows. Section~\ref{MySec2} contains main results.
    Section~\ref{MySec3} provides the proof of main results.

\section{Main results}   \label{MySec2}

    In this section we present our main results. Let
\[
    \mathcal{V}(s):={\frac{1}{\nu\Lambda(1-s)}}  \qquad \text{and} \qquad {J(s)}:={\frac{1-f'(s)}{\Lambda(1-s)}}-1.
\]

\begin{theorem}               \label{AATh2}
    If condition {\hyperref[RAfL]{$\left[{f_\Lambda}\right]$}} holds, then
\begin{itemize}
\item [\textbf{(1)}] the GF $U(s)$ has the following form:
\begin{equation}                 \label{2.1}
    U(s) = \mathcal{V}(s) - \mathcal{V}(0);
\end{equation}
\item [\textbf{(2)}] the derivative $U'(s)$ has the following  expression:
\begin{equation}                 \label{2.2}
    U'(s) = {J(s)}{\frac{\mathcal{V}{(s)}}{1-s}} \raise1.2pt\hbox{.}
\end{equation}
\end{itemize}
\end{theorem}

\begin{remark}
    Undoubtedly, the function $U(s)$, as the limit of the generating function, admits the form of a power series expansion
\[
    U(s)= \sum_{j \in\mathcal{S}}{u_j s^j},
\]
    where $u_j=\sum_{k\in{\mathcal{S}}}{u_kP_{kj}(1)}$ and $\sum_{k\in{\mathcal{S}}}{u_k p_{0}^{k}}=1$; see {\cite[Lemma~4]{Slack68}}.
    Then relation~\eqref{2.2} immediately implies that
\begin{equation}              \label{2.3}
    {u_1}={U'(0)}={\frac{J(0)}{\nu{p_0}}}={\frac{1-p_0-p_1}{\nu{p_0^2}}} \raise1.2pt\hbox{.}
\end{equation}
\end{remark}

    Next, differentiating the Slack’s altered definition {\hyperref[SAU]{$\left[{\mathcal{S}_{U}}\right]$}} we have
\[
    {U'_n(s)}=-{\frac{R'_n(s)}{Q_n\Lambda \bigl({Q_n}\bigr)}}\raise1.2pt\hbox{.}
\]
    Thus, we can interpret the statement of the Lemma~{\ref{AALem2}} in terms of the
    convergence ${U'_n(s)}\to{U'(s)}$ as $n\to\infty$. So we provide its refinement in the following theorem.

\begin{theorem}               \label{AATh3}
    If conditions {\hyperref[RAfL]{$\left[{f_\Lambda}\right]$}} and
    {\hyperref[RAO(L/x)]{$\left[\mathcal{R}_{\mathcal{L}}\right]$}} hold, then
\begin{equation}              \label{2.4}
    {U'_n(s)}={U'(s)}\left(1+\mathcal{O}\left(\frac{\,1\,}{n}\right)\right)  \qquad \text{as}\quad n \to \infty,
\end{equation}
    where $U'(s)$ has the form of \eqref{2.2}.
\end{theorem}

    The assertion of Theorem~{\ref{AATh3}} provides the following important limit result. Let
\[
    {\mathcal{N}_{\nu}(n)}:={\mathcal{L}^{-1/\nu}\left(\frac{1}{Q_n}\right)}
    \qquad \text{and} \qquad
    {\mathcal{P}^{\{j\}}_{\nu}(n)}:={(\nu{n})^{(1+\nu)/\nu}}{\textsf{p}_{j}{(n)}}.
\]

\begin{theorem}               \label{AATh4}
    If conditions {\hyperref[RAfL]{$\left[{f_\Lambda}\right]$}} and
    {\hyperref[RAO(L/x)]{$\left[\mathcal{R}_{\mathcal{L}}\right]$}} hold,
    then the sequence
\[
    \left\{\mathcal{P}_{\nu}(n):={\mathcal{P}^{\{1\}}_{\nu}(n)}\right\}
\]
    is SV at infinity such that
\begin{equation}              \label{2.5}
    {\frac{\mathcal{P}_{\nu}(n)}{\mathcal{N}_{\nu}(n)}} = {u_1} \cdot \left(1-{\frac{(1+\nu)^2}{2\nu^2}}
    \frac{\ln{n}}{n} + o\left(\frac{\ln{n}}{n} \right) \right)   \qquad \text{as}\quad n \to \infty,
\end{equation}
    where ${u_1}$ is given in \eqref{2.3}. Moreover
\[
    {\mathcal{N}_{\nu}(n)} \cdot \mathcal{L}^{{1/\nu}}
    \left({\frac{(\nu{n})^{{1/\nu}}}{\mathcal{N}_{\nu}(n)}} \right)
    \longrightarrow{1}   \qquad \text{as}\quad n \to \infty
\]
    and
\[
    {\mathcal{N}_{\nu}(n)}=C_{\mathcal{N}}+\mathcal{O}\left({n^{-\nu}}\right)  \qquad \text{as} \quad n \to \infty,
\]
    where $C_{\mathcal{N}}=C_{\mathcal{L}}^{-1/\nu}$ and $C_{\mathcal{L}}:=\mathcal{L}(\infty-)<\infty$.
\end{theorem}

    The statement of this theorem can be generalized for all ${j\in\mathcal{S}}$ as follows.

\begin{proposition}               \label{AAProp}
    If conditions {\hyperref[RAfL]{$\left[{f_\Lambda}\right]$}} and
    {\hyperref[RAO(L/x)]{$\left[\mathcal{R}_{\mathcal{L}}\right]$}} hold, then
\[
    {\frac{{\mathcal{P}^{\{j\}}_{\nu}(n)}}{\mathcal{N}_{\nu}(n)}}
    = {u_j} \cdot \bigl(1 + {\mho_n}\bigr),
\]
    where
\[
    \mho_{n} = -{\frac{(1+\nu)^2}{2\nu^2}} \frac{\ln{n}}{n} + o\left(\frac{\ln{n}}{n} \right)
    \qquad \text{as}\quad n \to \infty.
\]
\end{proposition}

    We leave the proof of {\hyperref[AAProp]{Propisition}} until our next works.

\section{Proof of results}   \label{MySec3}

   We will need the following auxiliary statement.

\begin{lemma}                 \label{AALem4}
    Let condition {\hyperref[RAfL]{$\left[{f_\Lambda}\right]$}} holds.
\begin{itemize}
\item [\textbf{1.}] Then
\begin{equation}              \label{3.1}
    \rho{(s)}:=\Big\vert \nu-{J(s)} \Big\vert \longrightarrow{0}  \qquad \text{as}\quad s \uparrow{1}.
\end{equation}
\item [\textbf{2.}] If, in addition {\hyperref[RAO(L/x)]{$\left[\mathcal{R}_{\mathcal{L}}\right]$}} holds, then
\begin{equation}              \label{3.2}
    \rho{(s)}=\mathcal{O}\left(\bigl(1-s\bigr)^{\nu}\right) \qquad \text{as}\quad s \uparrow{1}.
\end{equation}
\end{itemize}
\end{lemma}

\begin{proof}
    From representation {\hyperref[RAfL]{$\left[{f_\Lambda}\right]$}} we have
\begin{equation}              \label{3.3}
    1 - f'(s) = \Lambda (1-s) + (1-s)\Lambda{'}(1-s).
\end{equation}
    On the other hand, it was proved in the book {\cite[p.~401]{Bingham}} that
\[
    {\frac{y\Lambda'(y)} {\Lambda(y)}}\longrightarrow\nu  \qquad \text{as} \quad y\downarrow{0}.
\]
    Then it follows
\[
    {J(s)}={\frac{1-f'(s)}{\Lambda(1-s)}}-1 = {\frac{(1-s)\Lambda{'}(1-s)}{\Lambda(1-s)}}
    \longrightarrow {\nu}    \qquad \text{as}\quad s \uparrow{1}
\]
    which implies \eqref{3.1}.

    To prove the second part we first write
\begin{equation}                  \label{3.4}
    {\frac{y\Lambda'(y)} {\nu\Lambda(y)}} =  1 + \delta(y),
\end{equation}
    with some continuous $\delta(y)$ being that $\delta(y) \to 0$ as $y \downarrow 0$. And then we follow
    the corresponding arguments in {\cite[p.126]{ImomovSM23}}, relying, in contrast to them, on the condition
    {\hyperref[RAO(L/x)]{$\left[\mathcal{R}_{\mathcal{L}}\right]$}}. Then we obtain in this issue that
\[
    \delta (y) = \mathcal{O}\bigl({\Lambda(y)}\bigr)
    \qquad \text{as} \quad y\downarrow{0}.
\]
    Continuing discussions in accordance with {\cite[p.126]{ImomovSM23}}, we see that
    $C_{\mathcal{L}}:=\mathcal{L}(\infty-)<\infty$ and
\begin{equation}              \label{3.5}
    {\hyperref[RAO(L/x)]{\left[\mathcal{R}_{\mathcal{L}}\right]}} \quad \Longleftrightarrow \quad {\mathcal{L}(x)} =
    C_{\mathcal{L}} + \mathcal{O}\left({x^{-\nu}}\right)  \qquad \text{as} \quad x \to \infty.
\end{equation}
    Therefore it follows $\delta (y)=\mathcal{O}\left({y^{-\nu}}\right)$ as $y \downarrow 0$. Then using
    this conclusion, combining relations \eqref{3.3} and \eqref{3.4}, we get to the estimation \eqref{3.2}.

    The lemma is proved.
\end{proof}

\begin{proof}[\textbf{Proof of Theorem~\ref{AATh2}}]
    Put
\begin{equation}              \label{3.6}
    {\mathcal{M}}_n (s):=1-{\frac{\Lambda \bigl({R_n(s)}\bigr)}{\Lambda \bigl({Q_n}\bigr)}} \raise1.2pt\hbox{.}
\end{equation}
    Using relations \eqref{1.2} and \eqref{1.4}, in {\cite[p.131]{ImomovSM23}} proved that
\[
    n{\mathcal{M}}_n (s) \longrightarrow U(s)
    \qquad \text{as}\quad  n \to \infty.
\]
    Moreover, it was shown there {\cite[p.130]{ImomovSM23}} that
\begin{equation}              \label{3.7}
    {\lim_{n \to \infty}}\,{\frac{\,1\,}{{\nu}n}}\left[{\frac{1}{{\Lambda}\bigl({R_{n}(s)}\bigr)}}
    - {\frac{1}{{\Lambda}\bigl(1-s\bigr)}}\right] = 1.
\end{equation}
    Combining \eqref{3.6} and \eqref{3.7}, we obtain
\begin{eqnarray*}
    U(s)
    & = & {\lim_{n \to \infty}}n{\mathcal{M}}_n (s)     \nonumber\\
\nonumber\\
    & = &  {\lim_{n\to\infty}}n\left[1-{\frac{{\Lambda}\bigl(1-s\bigr)}{p_0}}
    {\frac{1-p_0\nu{n}} {{\Lambda}\bigl(1-s\bigr)\nu{n}+1}}\right]=\mathcal{V}(s)-\mathcal{V}(0).
\end{eqnarray*}
    We accounted for ${\Lambda}(1)={\mathcal{L}(1)}=p_0$ in the last step. The relation \eqref{2.1} is proved.

    The proof content of second part is short due to \eqref{3.3}. Write
\[
    U'(s)=\mathcal{V}'(s)={\frac{\Lambda'\left({1-s}\right)}{\nu\Lambda^{2}\left({1-s}\right)}}=
    {\frac{1-f'(s)-\Lambda\left({1-s}\right)}{\nu\left({1-s}\right)\Lambda^{2}\left({1-s}\right)}}  \raise1.2pt\hbox{.}
\]
    The right-hand side is easily transformed to the form of those part of \eqref{2.2}.

    The theorem is proved completely.
\end{proof}

\begin{remark}
    Repeatedly use of Abel equation \eqref{1.3}, with considering of relation \eqref{2.1}, yields
\[
    {\frac{1}{{\Lambda}\bigl({R_{n}(s)}\bigr)}} - {\frac{1}{{\Lambda}(1-s)}} = \nu{n}.
\]
    It more exact refines the well-known statement mentioned in \eqref{3.7}, indicating the absence
    of the limit operation as $n\to\infty$ on the left-hand side. Then under the condition
    {\hyperref[RAfL]{$\left[{f_\Lambda}\right]$}} it follows that
\[
    {Q_n}= {\frac{\mathcal{N}_{\nu}(n)}{\left(\nu{n}\right)^{1/\nu}}}
    \left(1-\frac{1}{p_0\nu{n}}\bigl(1+o(1)\bigr)\right) \qquad \text{as}\quad n \to \infty,
\]
    where ${\mathcal{N}_{\nu}(n)}={\mathcal{L}^{-1/\nu}\left({1}\big/{Q_n}\right)}$.
\end{remark}

\begin{proof}[\textbf{Proof of Theorem~\ref{AATh3}}]
    Repeatedly using \eqref{1.3} entails $U\bigl(f_n(s)\bigr)=U(s)+n$ and hence
\[
    f'_n(s)={\frac{U'(s)}{U'\bigl(f_n(s)\bigr)}}  \raise1.2pt\hbox{.}
\]
    Using relation \eqref{2.2} in last equality, in our notation we have
\begin{equation}              \label{3.8}
    {R'_n (s)}=-{\frac{J(s)}{J\bigl({R_n(s)}\bigr)}}\,
    {\frac{R_n(s)\Lambda \bigl({R_n(s)}\bigr)}{(1-s)\Lambda(1-s)}} \raise1.2pt\hbox{.}
\end{equation}
    We recall now to relation \eqref{2.2} and the Lemma~\ref{AALem4}. Then
\begin{align}                  \label{3.9}
    \left. \begin{array}{l}
    \displaystyle{J\bigl({R_n(s)}\bigr)}=\nu + \mathcal{O}\bigl(R^{\nu}_n(s)\bigr)
    =\nu+\mathcal{O}\left(\frac{\,1\,}{n}\right) \qquad \text{as}\quad n \to \infty  \\
\\
    \text{and} \\
\\
    \displaystyle{{\frac{J(s)}{\nu(1-s)\Lambda(1-s)}}=U'(s).}   \\
    \end{array} \right\}
\end{align}
    Formulas \eqref{3.8} and \eqref{3.9} complete the proof of the theorem.
\end{proof}

\begin{proof}[\textbf{Proof of Theorem~\ref{AATh4}}]
    First we rewrite \eqref{2.4} as follows:
\[
    {R'_n(s)}=-{U'(s)}{R_n(s)\Lambda \bigl({R_n(s)}\bigr)}
    \left(1+\mathcal{O}\left(\frac{\,1\,}{n}\right)\right)  \qquad \text{as}\quad n \to \infty.
\]
    Then, since ${R'_n(0)}=-\textsf{p}_{1}(n)$, letting $s=0$ implies
\begin{equation}                  \label{3.10}
    {\textsf{p}_{1}(n)}={U'(0)}{Q_n\Lambda \bigl({Q_n}\bigr)}
    \left(1+\mathcal{O}\left(\frac{\,1\,}{n}\right)\right)  \qquad \text{as}\quad n \to \infty.
\end{equation}
    Using Lemma~\ref{AALem3} we obtain
\begin{equation}                  \label{3.11}
    {{\Lambda}\bigl({Q_{n}}\bigr)}={\frac{1}{\nu{n}}}\left(1-{\frac{1+\nu}{2\nu}}{\frac{\ln{n}}{n}}
    +o\left(\frac{\ln{n}}{n}\right)\right)   \qquad \text{as}\quad n \to \infty
\end{equation}
    and
\begin{equation}                  \label{3.12}
    {Q_n}= {\frac{\mathcal{L}^{-1/\nu}\left(1\big/{Q_n}\right)}{(\nu{n})^{1/\nu}}}
    \left(1-{\frac{1+\nu}{2\nu^{2}}}{\frac{\ln{n}}{n}}+o\left(\frac{\ln{n}}{n}\right)\right)
    \qquad \text{as}\quad n \to \infty.
\end{equation}
    Further, combining \eqref{3.10}--\eqref{3.12} produces
\begin{equation}                  \label{3.13}
    {\textsf{p}_{1}{(n)}}={u_1}{\frac{\mathcal{N}_{\nu}(n)}{(\nu{n})^{(1+\nu)/\nu}}}
    \left(1-{\frac{(1+\nu)^2}{2\nu^2}}\frac{\ln{n}}{n}+o\left(\frac{\ln{n}}{n} \right) \right)
    \qquad \text{as}\quad n \to \infty,
\end{equation}
    where ${\mathcal{N}_{\nu}(n)}={\mathcal{L}^{-1/\nu}\left(1\big/{Q_n}\right)}$ and $u_1$ is defined in \eqref{2.3}.
    It is known that ${Q_n}=\mathcal{N}(n)\big/{(\nu{n})^{1/\nu}}$ which is a result of Lemma~\ref{AALem1},
    where $\mathcal{N}(\ast)$ is SV at infinity with the asymptotic property \eqref{1.5}.
    In accordance with this property we write that
\[
    1={\mathcal{N}_{\nu}(n)}\cdot\mathcal{L}^{{1/\nu}}\left({\frac{1}{Q_n}} \right)=
    {\mathcal{N}_{\nu}(n)}\cdot\mathcal{L}^{{1/\nu}}\left({\frac{(\nu{n})^{{1/\nu}}}{\mathcal{N}(n)}}\right)
    \sim{\frac{\mathcal{N}_{\nu}(n)}{\mathcal{N}(n)}}
    \qquad \text{as}\quad n \to \infty.
\]
    Then it follows
\[
    {\mathcal{N}_{\nu}(n)} \cdot \mathcal{L}^{{1/\nu}} \left({\frac{(\nu{n})^{{1/\nu}}}{\mathcal{N}_{\nu}(n)}} \right)
    \longrightarrow{1}   \qquad \text{as}\quad n \to \infty.
\]
    But by virtue of \eqref{3.5}
\[
    {\mathcal{N}_{\nu}(n)}= C_{\mathcal{N}} + \mathcal{O}\left({n^{-\nu}}\right)  \qquad \text{as} \quad n \to \infty,
\]
    where $C_{\mathcal{N}}=C_{\mathcal{L}}^{-1/\nu}$. Recalling now denotation
    ${\mathcal{P}_{\nu}(n)}:={(\nu{n})^{(1+\nu)/\nu}}{\textsf{p}_{1}{(n)}}$,
    we transform the asymptotic relation \eqref{3.13} to the form of \eqref{2.5}.

    The proof is completed.
\end{proof}



\begin{thebibliography}{99}

\bibitem{ANey}
    {\sc Athreya, K.~B. and  Ney, P.~E.} (1972). {\em Branching processes.} Springer, New York.

\bibitem{Bingham}
    {\sc Bingham, N.~H., Goldie, C.~M. and Teugels, J.~L.} (1987). {\em Regular variation.} Cambridge.

\bibitem{Harris63}
    {\sc Harris, T.~E.} (1963). {\em The theory of branching processes}, Springer-Verlag.

\bibitem{ImomovTukhtaev19}
    {\sc Imomov, A.~A. and Tukhtaev, E.~E.} (2019). On application of slowly varying functions with
    remainder in the theory of Galton-Watson branching process.
    {\em Jour. Siber. Fed. Univ.: Math. Phys.} \textbf{12(1)}, 51--57.

\bibitem{Imomov19}
    {\sc Imomov, A.~A.} (2019). On a limit structure of the Galton-Watson branching processes
    with regularly varying generating functions.
    {\em Prob. and math. stat.} \textbf{39(1)}, 61--73.

\bibitem{ImomovSM23}
    {\sc Imomov, A.~A. and Tukhtaev, E.~E.} (2023).On asymptotic structure of critical Galton-Watson branching processes
    allowing immigration with infinite variance. {\em Stochastic Models} \textbf{39(1)}, 118--140, DOI: 10.1080/15326349.2022.2033628.

\bibitem{KNS}
    {\sc Kesten, H. Ney, P.~E. and Spitzer, F.~L.} (1966). The Galton–Watson process with mean one and finite variance.
    {\em Jour. Appl. Prob.} \textbf{11(4)}, 579--611.

\bibitem{Seneta67}
    {\sc Seneta, E.} (1967). The Galton-Watson process with mean one.
    {\em Jour. Appl. Prob.} \textbf{4}, 489--495.

\bibitem{Slack68}
    {\sc Slack, R.~S.} (1968). A branching process with mean one and possible infinite variance.
    {\em Wahrscheinlichkeitstheor. und Verv. Geb.} \textbf{9}, 139--145.

\end{thebibliography}
\end{document}